\documentclass[11pt]{article}
 \usepackage[utf8]{inputenc}
 \usepackage[T1]{fontenc}

\usepackage{amsmath,amstext,amsthm,amssymb,amscd,version}
\usepackage{breqn}
  \usepackage[all]{xy}

\def\Z{\mathbb{Z}}

\def\R{\mathbb{R}}
\def\C{\mathbb{C}}
\def\P{\mathbb{P}}

\def\d{\partial}

\newtheorem{thm}{Theorem}
\newtheorem*{thm*}{Theorem}

\newtheorem{lem}{Lemma}[section]
\newtheorem{prop}[lem]{Proposition}

\newtheorem{conj}[lem]{Conjecture}
\theoremstyle{definition}
\newtheorem{rem}[lem]{Remark}
\newtheorem{exa}[lem]{Example}
\newtheorem{defn}[lem]{Definition}
\newtheorem{assu}{Assumption}

\newcommand{\comm}[1]{}

\def\Aa{\mathcal{A}}

\DeclareMathOperator{\vol}{vol}
\DeclareMathOperator{\Vol}{Vol}
\DeclareMathOperator{\Gr}{Gr}
\DeclareMathOperator{\inj}{inj}
\DeclareMathOperator{\End}{End}
\DeclareMathOperator{\Hom}{Hom}
\DeclareMathOperator{\rk}{rk}
\DeclareMathOperator{\Sp}{Sp}
\DeclareMathOperator{\Tr}{Tr}

\title{Lyapunov exponents of the Brownian motion on a K\"ahler manifold}
\author{Jeremy Daniel and Bertrand Deroin}
\date{}

\begin{document}

\maketitle

\abstract{If $E$ is a flat bundle of rank $r$ over a K\"ahler manifold $X$, we define the Lyapunov spectrum of $E$: a set of $r$ numbers controlling the growth of flat sections of $E$, along Brownian trajectories. We show how to compute these numbers, by using harmonic measures on the foliated space $\mathbb{P}(E)$. Then, in the case where $X$ is compact, we prove a general inequality relating the Lyapunov exponents and the degrees of holomorphic subbundles of $E$ and we discuss the equality case.}

\section*{Introduction}

Let $(X,g)$ be a K\"ahler manifold of dimension $d$ and let $(E,h)$ be a complex flat vector bundle of rank $r$ over $X$. If $X$ and $E$ satisfies certain assumptions of bounded geometry, then Kingman subadditive theorem shows that the growth of flat sections along Brownian trajectories on $X$ is controlled by a number $\lambda$, called the first Lyapunov exponent of $E$. More generally, considering the exterior powers of $E$, one can define $r$ numbers $\lambda_1 \geq \dots \geq \lambda_r$. This set is called the Lyapunov spectrum of $E$ and can also be defined by an application of Oseledets multiplicative ergodic theorem. By the symmetry of the Brownian motion, the spectrum is symmetric with respect to $0$, namely $\lambda_{r-k} = -\lambda_k$.

Pioneer work \cite{KoZor}, formalized in \cite{Forni}, shows that for variations of Hodge structures of weight $1$ over curves of finite type, the sum of the positive Lyapunov exponents equals the degree of the Hodge bundle, up to some normalization. Here, the dynamics on $X$ is given by the geodesic flow, rather than the Brownian motion. This formula has been further studied in \cite{EskKoZor_cyclic}, \cite{EskKoZor_teichmuller}; it has been generalized in \cite{Filip} for certain variations of Hodge structures of weight $2$ and in \cite{KapMol} over complex hyperbolic manifolds of higher dimension.

In fact, it has been observed by the second author -- in his PhD thesis (see e.g. \cite{Der_Levi}) and more recently in his work with R. Dujardin \cite{DerDuj} -- that in the case of a rank $2$ bundle over a curve $C$, there is a cohomological interpretation for the top Lyapunov exponent $\lambda_1$. Assume for the sake of simplicity that $C$ is a compact curve and consider the projectivized bundle $\mathbb{P}(E)$ over $C$. This space carries a foliation by curves, which is induced by the flat structure on $E$. One can show that there exists a harmonic current $T$ on $\mathbb{P}(E)$ of bidimension $(1,1)$, which is positive on the leaves of the foliation. Such a current is not closed in general but still defines a homology class in the dual of $H^2(X,\C)$, thanks to the equality of Bott-Chern and de Rham cohomologies on the K\"ahler manifold $\mathbb{P}(E)$. The formula
\begin{equation}\label{lambda1}  \lambda_1 = \frac{[T] \cdot \mathcal{O}(1)}{[T]\cdot [\omega]} \end{equation}
then holds, where $\mathcal{O}(1)$ is the anti-canonical line bundle on $\mathbb{P}(E)$ and $\omega$ is the pullback of the K\"ahler form. 

The goal of this work is threefold. First, we generalize the above formalism to a K\"ahler manifold of higher dimension and to a flat bundle of higher rank: for every $k=1,\dots,r$, we construct a pluriharmonic current $T_k$ of bidimension $(1,1)$ on the Grassmann bundle $\Gr(k,E)$, which is positive on the leaves of the foliation induced by the flat structure. Denoting by $\mathcal{O}(1)$ the anti-tautological line bundle on $\Gr(k,E)$ -- if $P$ is in $\Gr(k,E)$, then $\mathcal{O}(1) = \Lambda^k P^\ast$ at the point $P$ -- we show that
\begin{equation}\label{eq: Lyapunov exponent}  \lambda_1 + \ldots + \lambda_k   = T_k \cdot c_1(\mathcal{O}(1)), \end{equation}
where the Chern form is taken with respect to some metric induced by $h$. If $X$ is compact, this equality is purely cohomological.

Then, we consider the case where $X$ is compact and we recover and complete a result of \cite{DerDuj} ($r=2$ and $d=1$) and \cite{EKMZ} ($d=1$) showing that the Lyapunov spectrum satisfies the following estimates. For any holomorphic subbundle $F$ of $E$ of co-rank $k$, we have that 
\begin{equation} \label{eq:minoration} \sum_{l=1}^k  \lambda_l \geq \pi \cdot \deg(F). \end{equation}

Finally, we interpret the difference between the right and left hand sides of \eqref{eq:minoration}, as the intersection $[T_k] \cap [D_F]$, where $D_F\subset Gr(k,E)$ is the divisor of $k$-planes that intersect $F$ non trivially. At the end, we obtain that equality in formula \eqref{eq:minoration} happens if and only if the limit set of the monodromy representation in $\mathbb P (E)$ does not meet $D_F$. 

In the case where $X$ is non-compact, everything should still hold with some general assumptions of bounded geometry but some technical difficulties are still unsolved. We nevertheless discuss the case of some monodromy representations over the Riemann sphere minus $3$ points, considered in paragraph 6.3 of \cite{EKMZ}; they are related to the hypergeometric equation. Apart from the non-compactness issue, we show that, if the monodromy is thick, then the inequality in \eqref{eq:minoration} is strict. This leaves the case of thin monodromies open: showing that the equality then holds seems to us one of the most challenging problems in the topic.

\tableofcontents

\newpage

\section{The Lyapunov spectrum of a flat bundle}

In this section, we define and study the basic properties of the Lyapunov exponents of a flat bundle over a Riemannian manifold. The dynamics is defined by considering Brownian motion on the manifold and we assume that the reader is familiar with its definition and basic notions, as exposed e.g. in \cite{Hsu}.

The results will be used later when the basis manifold is a compact K\"ahler manifold. Since this does not need extra work, we make the study in a more general setting.

\subsection{Brownian motion in Riemannian geometry}

Let $(X,g)$ be a connected $d$-dimensional complete Riemannian manifold. We denote by $\pi: \tilde{X} \rightarrow X$ its universal cover. 

\begin{assu} \label{bound_geom}
We make the following assumptions of \emph{bounded geometry} on $X$:
 \begin{itemize}
 \item $X$ has finite volume;
 \item the Ricci curvature of $X$ is uniformly bounded from below. 
\end{itemize}
\end{assu}

Let $W(X)$ be the space of continuous paths $\gamma:[0,\infty[ \rightarrow X$, with its structure of filtered measurable space. Given any probability measure $\mu$ on $X$, there exists a unique $\Delta_X/2$-diffusion measure $\mathbb{P}_\mu$ on $W(X)$ with initial distribution $\mu$, where $\Delta_X$ denotes the Laplacian operator on $X$ (see e.g. \cite{Hsu}, p.79). If $\mu$ is a Dirac measure $\delta_x$, we write $\mathbb{P}_x$ instead of $\mathbb{P}_{\delta_x}$. 

We recall Dynkin's formula: 
\begin{equation}\label{Dynkin}
 \mathbb{E}_x [f(\gamma_t)] = f(x) + \frac{1}{2} \mathbb{E}_x \Big[\int_0^t \Delta_X f(\gamma_s) \Big] ds
\end{equation}
for any test function $f \in \mathcal{C}_c^\infty(X)$.
 
Let $p_X(t,x,y)$ be the (minimal) heat kernel on $X$; it is the transition density function of Brownian motion. In other words,
\begin{equation}
 \mathbb{E}_x [f(\gamma_t)] = \int_X p_X(t,x,y) f(y)dy,
\end{equation}
for any $f \in L^1(X,g)$. Here, $dy$ stands for the measure on $X$ induced by the metric $g$.

\begin{rem} \label{stoch_compl}
 By Theorem $4.2.4$ in \cite{Hsu}, a complete Riemannian manifold whose Ricci curvature is uniformly bounded from below is \emph{stochastically complete}, meaning that
\begin{equation}
 \int_X p_X(t,x,y) dy = 1,
\end{equation}
for any $(t,x) \in ]0,+\infty| \times X$. Hence, we can forget about Brownian paths exploding in finite time.
\end{rem}

This discussion can also be performed on $\tilde{X}$; we will use similar notations.

\begin{rem}\label{corr_univ}
 Let $\tilde{x}$ be a point in $\tilde{X}$, living in the fiber of $\pi$ above a point $x$ in $X$. The spaces $W(\tilde{X})$ and $W(X)$ are endowed with probability measures $\mathbb{P}_{\tilde{x}}$ and $\mathbb{P}_x$. The subspaces $W_{\tilde{x}}(\tilde{X})$ and $W_x(X)$ of paths starting at $\tilde{x}$ and $x$ have total mass and correspond bijectively by the map $\pi$; moreover the probability measures $\mathbb{P}_{\tilde{x}}$ and $\mathbb{P}_x$ are equal, under this correspondence.
\end{rem}

Let $\theta_X^t$ denote the \emph{shift} in $W(X)$:
\begin{equation*}
 (\theta_X^t) \gamma(s) := \gamma(s+t).
\end{equation*}

We will always consider the probability measure $\mu$ on $X$, defined by normalizing the volume form:

$$ d\mu(y) = \frac{1}{\vol(X)} dy.$$

The shift is measure-invariant and ergodic for the measure $\mathbb{P}_\mu$.

\subsection{The cocycle of a flat bundle}

Let $(E,D)$ be a complex flat vector bundle of rank $n$ over $X$. Let $h$ be a smooth Hermitian metric on $E$; we denote by $<\cdot,\cdot>$ the associated inner product. There is a unique decomposition
\begin{equation*}
 D = \nabla + \alpha,
\end{equation*}
where $\nabla$ is a metric connection and $\alpha$ is a $1$-form with values in the space $\End_h(E)$, of Hermitian endomorphisms of $E$. We make the following compatibility assumption between the metric and the flat connection:

\begin{assu}\label{bound_higgs}
 The operator norm of $\alpha$, relatively to $g$ and $h$, is uniformly bounded on $X$.
\end{assu}

We define
\begin{equation*}
H^t(\gamma) := \log ||P_{\gamma,0,t}||,
\end{equation*}
where $\gamma$ is in $W(X)$ and $||P_{\gamma,0,t}||$ is the operator norm of the parallel transport along $\gamma_{|[0,t]}$. If $\gamma$ is a path in $W(X)$, the operator norms satisfy 
$$
||P_{\gamma,0,t+t'}|| \leq ||P_{\gamma,0,t}|| \, ||P_{\gamma,t,t+t'}||,
$$
for every $t,t' \geq 0$. Hence, $(H^t)_{t \in ]0,+\infty|}$ is a \emph{subadditive cocycle} in $(W(X),\theta^t)$, meaning that
\begin{equation}
 H^{t+t'}(\gamma) \leq H^t(\gamma) + H^{t'}(\theta^t \gamma).
\end{equation}

In order to define the Lyapunov exponent of $(E,D,h)$, we need the following technical result:

\begin{thm} \label{H1_L1}
 For any $t_0 > 0$, the function $\sup_{t \in [0,t_0]} |H^t|$ is in $L^1(W(X),\mathbb{P}_\mu)$.
\end{thm}

\begin{proof}
 We claim that there exists a positive constant $C$ such that, for any $x$ and $y$ in $\tilde{X}$,
\begin{equation} \label{par_trans}
 \Big|\log ||P_{x,y}||\,\Big| \leq C.d_{\tilde{X}}(x,y),
\end{equation}
where $||P_{x,y}||$ is the operator norm of the parallel transport from $x$ to $y$ on $\pi^{\ast} E$.

Indeed, let $\omega:[0,T] \rightarrow \tilde{X}$ be a geodesic from $x$ to $y$. If $u$ is in the fiber $E_x$,
we write $u(t)$ for its parallel transport along $\omega$. We choose such a $u$ satisfying $||P_{x,y}|| = \frac{h(u(T))}{h(u(0))}$ 
and we consider the function $f(t) = \log h(u(t))$. Since $u(t)$ is flat,
\begin{eqnarray*}
 \frac d{dt} h(u(t))^2 &=& 2 <\nabla_{\dot{\omega}_t} u(t), u(t)> \\
&=& -2<\alpha(\dot{\omega}_t) u(t),u(t)>.
\end{eqnarray*}
Hence,
\begin{equation*}
 f'(t) = -\frac{<\alpha(\dot{\omega}_t) u(t),u(t)>}{h(u(t))^2}.
\end{equation*}
By Assumption \ref{bound_higgs}, this implies $|f'(t)| \leq C |\dot{\omega}_t|$, for some positive constant $C$. Integrating from $0$ to $T$, 
we get $\big|\log \frac{h(u(T))}{h(u(0))}\big| \leq C \int_0^T |\dot{\omega}_t| dt = C \dot d_{\tilde{X}}(x,y)$, proving the claim.

We define functions $H_{\tilde{X}}^t$ on $W(\tilde{X})$ in the same way than the functions $H^t$ on $W(X)$. Then,
\begin{eqnarray}
\mathbb{E}_{\mu}(\sup_{t \in [0,t_0]} |H^t|) &=& \int_X \mathbb{E}_x(\sup_{t \in [0,t_0]} |H^t|) d\mu(x) \nonumber \\
&=& \int_X \mathbb{E}_{\tilde{x}}(\sup_{t \in [0,t_0]} |H_{\tilde{X}}^t|) d\mu(x) \nonumber,
\end{eqnarray}
where $\tilde{x}$ is an arbitrary point in $\tilde{X}$ over $x$. We have just proved that 
$$ |H_{\tilde{X}}^t(\gamma)| \leq C.d(\tilde{x},\gamma(t)),$$
if $\gamma$ is a path in $W(\tilde{X})$ starting at $\tilde{x}$. It is thus sufficient to show that 
$$ \mathbb{E}_{\tilde{x}}\Big(\sup_{t \in [0,t_0]} d(\tilde{x},\gamma(t))\Big)$$
is bounded, uniformly in $\tilde{x}$.

Here, we quote equation $(8.65)$ in \cite{Stroock}. It says that for every $t_0$, there is a constant $C(t_0)$ depending only on
the dimension of $X$ and the bound on the Ricci curvature such that, for every $\tilde{x}$ in $\tilde{X}$ and radius $r > 0$
$$ \mathbb{P}_{\tilde{x}} (\sup_{t \in [0,t_0]} d(\tilde{x},\gamma(t)) \geq r) \leq \sqrt{2} \exp\big(-\frac{r^2}{4t_0} + C(t_0)\big).$$
Since the left hand side is integrable as a function of $r$, this concludes the proof.
\end{proof}

\subsection{Lyapunov exponents}

Since $(H^t)_{t \in ]0,+\infty[}$ is a subadditive cocycle on $(W(X),\theta^t)$ satisfying $\sup_{[t \in [0,1]} |H^t| \in L^1(W(X), \mathbb{P}_\mu)$, Kingman subadditive theorem implies:
\begin{prop}
 $(H_t/t)_t$ converges $\mathbb{P}_\mu$-almost everywhere to some constant $\lambda$, when $t$ goes to $+\infty$.
\end{prop}

The limit is a constant by ergodicity of the shift $(\theta^t)_t$. 

\begin{defn}
 We call $\lambda$ the (first) \emph{Lyapunov exponent} of $(E,D,h)$.
\end{defn}

The rest of this section is devoted to the proof of some basic and well-known properties of Lyapunov exponents.

\begin{prop}\label{dep_metric}
 Let $h$ and $h'$ be mutually bounded Hermitian metrics on $E$: i.e. there exists a constant $C \geq 1$ such that
\begin{equation*}
 C^{-1} \leq \frac{h_x(u)}{h'_x(u)} \leq C
\end{equation*}
holds, for any $x$ in $X$ and non-zero $u$ in $E_x$. The cocycle $(H_t)$, computed with $h$, satisfies the integrability condition of Theorem \ref{H1_L1} if and only if the same is true for the cocycle computed with $h'$. Moreover, the Lyapunov exponents $\lambda(h)$ and $\lambda(h')$ are then equal.
\end{prop}

\begin{proof}
 If $\tilde{x},\tilde{y}$ are in $\tilde{X}$, the operator norms of $P_{\tilde{x},\tilde{y}}$ with respect to $h$ and $h'$ satisfy
$$ C^{-2} \leq \frac{||P_{\tilde{x},\tilde{y}}||_{h'}}{||P_{\tilde{x},\tilde{y}}||_{h'}} \leq C^2.$$ 
The first point of the proposition easily follows from the proof of Theorem \ref{H1_L1}. Moreover, the Lyapunov exponents satisfy
\begin{eqnarray*}
|\lambda(h)-\lambda(h')| &\leq& \lim_{t \rightarrow + \infty} \frac 1t \int_X \Big( \int_{\tilde{X}} 2 \log C \, p^{\tilde{X}}(t,\tilde{x},\tilde{y}) d\tilde{y} \Big) d\mu(x) \\
&=& \lim_{t \rightarrow + \infty}\frac{2}t \log C \\
&=& 0,
\end{eqnarray*}
which concludes the proof.
\end{proof}

\begin{rem}
 This statement is weak but is sufficient in the case where $X$ is compact. For a more general result, see e.g. Theorem A.5 in \cite{EKMZ}.
\end{rem}

\begin{lem}
The Lyapunov exponent is nonnegative and vanishes if the rank of $E$ is $1$.
\end{lem}

\begin{proof}
It is convenient to consider the space $X^{\pi_1(X)} := (\tilde{X} \times \tilde{X})/{\pi_1(X)}$, where the action of $\pi_1(X)$ is the diagonal one. This space carries a natural metric, since the metric on $\tilde{X} \times \tilde{X}$ is invariant by the action and
\begin{eqnarray*}
 \mathbb{E}_{\mu}(H^t) &=& \int_X \Big( \int_{\tilde{X}} \log ||P_{\tilde{x},\tilde{y}}||\, p^{\tilde{X}}(t,\tilde{x},\tilde{y}) d\tilde{y} \Big) d\mu(x) \\
&=& \frac{1}{\vol(X)} \int_{X^{\pi_1(X)}} \log ||P_{\tilde{x},\tilde{y}}||\, p^{\tilde{X}}(t,\tilde{x},\tilde{y}) d\tilde{x} \otimes d\tilde{y},
\end{eqnarray*}
since $ \log ||P_{\tilde{x},\tilde{y}}||\, p^{\tilde{X}}(1,\tilde{x},\tilde{y})$ is invariant for the diagonal action of $\pi_1(X)$. By symmetry of the heat kernel, we get:
\begin{equation} \label{symm_form}
 2 \,\mathbb{E}_{\mu}(H^t) = \frac{1}{\vol(X)} \int_{X^{\pi_1(X)}} \big(\log ||P_{\tilde{x},\tilde{y}}|| + \log ||P_{\tilde{y},\tilde{x}}|| \big) \, p^{\tilde{X}}(t,\tilde{x},\tilde{y}) d\tilde{x} \otimes d\tilde{y}.
\end{equation}
Since $P_{\tilde{x},\tilde{y}} . P_{\tilde{y},\tilde{x}}$ is the identity, the inequality $1 \leq ||P_{\tilde{x},\tilde{y}}||.||P_{\tilde{y},\tilde{x}}||$ holds. Morever, it is an equality for a rank $1$ vector bundle $E$. Since $\lambda$ is the limit of $\frac {\mathbb{E}_{\mu}(H^t)}t$, we get $\lambda \geq 0$, with equality in the rank $1$ case.
\end{proof}

\paragraph{Higher Lyapunov exponents} For any $r$ from $1$ to $n = \rk(E)$, we consider the vector bundle $\Lambda^r E$. It is endowed with a flat connection $D^{\Lambda^r E}$ and a Hermitian metric $h^{\Lambda^r E}$, induced from the connection and metric on $E$. We define recursively the higher Lyapunov exponents in the following way:

\begin{defn}\label{def_higher}
We denote by $\lambda(\Lambda^r E)$ the first Lyapunov exponent of $\Lambda^r(E)$. The \emph{Lyapunov exponents} $(\lambda_i)_{i=1}^n$ of $E$ are defined by
\begin{equation*}
 \lambda_1 + \dots + \lambda_r = \lambda(\Lambda^r(E)).
\end{equation*}
The $n$-tuple $(\lambda_1,\dots,\lambda_n)$ is the \emph{Lyapunov spectrum} of $E$.
\end{defn}

\begin{rem}
 The Lyapunov exponent of $\Lambda^r E$ is well defined since Assumption \ref{bound_higgs} on $\Lambda^r E$ is satisfied if it is satisfied on $E$.
\end{rem}

\begin{prop}\label{symm_lyapunov}
 The Lyapunov spectrum of $E$ is symmetric, that is $\lambda_r = - \lambda_{n-r}$, for every $r = 1,\dots,n$.
\end{prop}

\begin{proof}
 Let $r$ be an integer between $1$ and $n$. If $u$ and $v$ are points in $\tilde{X}$, we write $\Lambda^i P_{u,v}$ for the parallel transport on the vector bundle $\Lambda^i E$, from $u$ to $v$. A similar computation to the one giving equation \eqref{symm_form} shows that:
\begin{eqnarray*}
 \lambda_r + \lambda_{n-r} &=& \lim_{t \rightarrow +\infty} \frac 1t \int_{X^{\Gamma}} \Big(\log \frac{||\Lambda^r P_{\tilde{x},\tilde{y}}||}{||\Lambda^{r-1} P_{\tilde{x},\tilde{y}}||} \\
&+& \log \frac{||\Lambda^{n-r} P_{\tilde{y},\tilde{x}}||}{||\Lambda^{n-r-1} P_{\tilde{y},\tilde{x}}||} \Big)\, p^{\tilde{X}}(t,\tilde{x},\tilde{y}) d\tilde{x} \otimes d\tilde{y}.
\end{eqnarray*}
The statement then follows from the following linear algebra lemma.
\end{proof}

\begin{lem}
 Let $M$ be an automorphism of a Hermitian vector space $(E,h)$ of dimension $n$. We endow $\Lambda^i E$ with the induced metric. One has the following equality of operator norms, for any integer $r$ between $1$ and $n$:
\begin{equation}\label{eqn_lambda_r}
 \log \frac{||\Lambda^r M||}{||\Lambda^{r-1} M||} = - (\log \frac{||\Lambda^{n-r} M^{-1}||}{||\Lambda^{n-r-1} M^{-1}||}).
\end{equation}
\end{lem}

\begin{proof}
 By the $KAK$ decomposition, one can assume that $M$ is a diagonal matrix, with positive entries $a_1 \geq \dots \geq a_n$. Then $\Lambda^i M$ is diagonal with entries the products $\prod_I a_I$, where $I$ is a subset of cardinal $i$ in $\{1,\dots,n\}$. The operator norm of $\Lambda^i M$ is $\prod_{j=1}^i a_j$ and the operator norm of $\Lambda^i M^{-1}$ is $\prod_{j=1}^i a_{n-j}^{-1}$. Hence, the left hand side of equation \eqref{eqn_lambda_r} is $\log (a_r)$ and the right hand side is $- \log (a_r^{-1})$. This concludes the proof of the lemma.
\end{proof}

\begin{rem}
 The proof also shows that the Lyapunov spectrum is non-increasing: $\lambda_r \geq \lambda_{r+1}$, for any $r$ between $1$ and $n-1$.
\end{rem}

\section{Harmonic measures on $\mathbb{P}(E)$}

In this section, we use ergodicity of the Brownian motion to compute the Lyapunov exponent of $(E,D)$ as an integral in space. Assumptions \ref{bound_geom} and \ref{bound_higgs} are still satisfied.

\subsection{Existence of harmonic measures}

Let $p: \mathbb{P}(E) \rightarrow X$ be the projectivized bundle of $E$. Since $D$ is a flat connection on $E$, $\mathbb{P}(E)$ carries a $d$-dimensional foliation, where $d$ is the dimension of $X$. On the leaves of the foliation, we consider the Brownian motion with respect to the metric $g$. We define the \emph{heat semigroup} $(\mathcal{P}_t)$ on $\mathbb{P}(E)$ by
$$ (\mathcal{P}_t f)(x) = \int_{\mathcal{L}_x} f(y) p^{\mathcal{L}_x}(t,x,y) dy,$$
where $f$ is any bounded measurable function on $\mathbb{P}(E)$, $\mathcal{L}_x$ is the leaf through the point $x$ and $p^{\mathcal{L}_x}$ is the heat kernel on $\mathcal{L}_x$. A fundamental property of this semigroup is the following:

\begin{prop}\label{Feller}
The heat semigroup has the \emph{Feller property}: if $f$ is a continuous bounded function, then $\mathcal{P}_t f$ is a continuous bounded function.
\end{prop}

\begin{proof}
We denote by $\pi_E : \mathbb{P}(\pi^\ast E) \rightarrow \mathbb{P}(E)$ the projectivized pullback bundle over $\tilde{X}$. Let $(\mathcal{P}^{\tilde{X}}_t)$ be the heat semigroup on $\mathbb{P}(\pi^\ast E)$. Since $\tilde{X}$ is simply connected, the foliation on $\mathbb{P}(\pi^\ast E)$ is a product and the Feller property of $(\mathcal{P}^{\tilde{X}}_t)$ comes from the regularity of the solution of the heat equation, with respect to the initial condition. More precisely, if we equip $C^0 (\tilde{X},\R)$ with the topology of uniform convergence on compact subsets, then the heat flow preserves the closed subset of fonctions uniformly bounded by a certain positive constant, and acts continously on this set. 

Moreover, the semigroups satisfy
$$ \mathcal{P}^{\tilde{X}}_t (\pi_E \circ f) = \pi_E \circ \mathcal{P}_t(f),$$
for any bounded measurable $f$ on $\mathbb{P}(E)$. If $f$ is continuous, then the left hand side is continuous; since $\pi_E$ is a local homeomorphism, $\mathcal{P}_t(f)$ itself is continuous.
\end{proof}

The heat semigroup $(\mathcal{P}_t)$ acts dually on the space $\mathcal{M}_1(\mathbb{P}(E))$ of probability measures by
$$ \mathcal{P}_t(\nu)(f) = \nu(\mathcal{P}_t(f)),$$
for any continuous bounded function $f$.

\begin{defn}
 A probability measure $\nu$ on $\mathbb{P}(E)$ is \emph{harmonic} if $\mathcal{P}_t(\nu) = \nu$, for every $t \geq 0$.
\end{defn}

The following theorem is folklore. In this setting, it was first proved in \cite{Garn}, in the compact case.

\begin{thm} \label{exist_harm}
 There exists a harmonic measure $\nu$ on $\mathbb{P}(E)$.
\end{thm}

\begin{proof}
 Let $\nu_0$ be a probability measure on $\mathbb{P}(E)$ whose push-forward $p_\ast \nu_0$ on $X$ is the probability measure $\mu$. We claim that the set $(\mathcal{P}_t(\nu_0))_{t \in \R}$ of probability measures is tight. Indeed, let $\epsilon > 0$ and let $K \subset X$ be a compact subspace such that $\mu(K) \geq 1-\epsilon$. Since $p$ is proper, $p^{-1}(K)$ is a compact subspace of $\mathbb{P}(E)$. Moreover, the push-forward of every $\mathcal{P}_t(\nu_0)$ is $\mu$, since $\mu$ is invariant by heat diffusion. This shows that $(\mathcal{P}_t(\nu_0))(p^{-1}(K)) \geq 1-\epsilon$, for every $t$ and concludes the claim.

We define new probability measures $\mu_t$ on $\mathbb{P}(E)$ by
$$ \mu_t := \frac 1t \int_0^t \mathcal{P}_s(\nu_0) ds,$$
for any positive $t$. The set $(\mu_t)_{t > 0}$ is also tight. By Prokhorov theorem, there exists a sequence $t_n \rightarrow +\infty$ such that $\mu_{t_n}$ has a weak limit $\mu_\ast$. Using the Feller property, one shows that $\mu_\ast$ is harmonic; see e.g. the proof of Theorem 6.1 (Krylov-Bogolioubov) in \cite{Hairer}.
\end{proof}

\begin{defn}
 A subset $\mathcal{M}$ of $\mathbb{P}(E)$ is called \emph{invariant} if it is an union of leaves.
\end{defn}

\begin{rem} \label{supp_harm}
 Suppose that $\mathcal{M}$ is a non-empty closed invariant subset of $\mathbb{P}(E)$. Then one can begin with a measure $\nu_0$ with support in $\mathcal{M}$. This shows that there exists a harmonic measure with support in $\mathcal{M}$.
\end{rem}

\begin{rem} \label{rem_mes_harm}
Let $f$ be a continuous function with compact support on $\mathbb{P}(E)$, which is smooth in the direction of the leaves. By definition of the heat semigroup, $\frac{1}{t}(\mathcal{P}_t f - f)$ converges pointwise to the foliated Laplacian $\Delta_{\mathcal{F}} f$, when $t$ goes to zero. By arguments similar to the proof of Proposition \ref{Feller} or Fact 1 in \cite{Garn}, one shows that this convergence is uniform. It follows that, for a harmonic measure $\nu$, $\nu(\Delta_{\mathcal{F}} f) = 0$, for any such $f$. The converse holds, at least if $X$ is compact, but requires more work: see \cite{Garn}, Fact 4. We will not use this fact in the following.
\end{rem}

\subsection{Local structure}

Following \cite{Garn}, we give a local picture of harmonic measures. The fibers of the map $p: \mathbb{P}(E) \rightarrow X$ are transverse to the foliation. If $x$ belongs to $X$, there is a neighborhood $U$ of $x$ in $X$ such that parallel transport gives a diffeomorphism $p^{-1}(U) \cong U \times \mathbb{P}(E_x)$. 

With this identification, the harmonic measure $\nu$ on the neighborhood $p^{-1}(U)$ of $x$ disintegrates in the following way: there is a transversal finite measure $\gamma$ on $\mathbb{P}(E_x)$ and a nonnegative bounded measurable function $\phi$ on $U \times \mathbb{P}(E_x)$ such that $\phi$ is leaf-harmonic for $\gamma$-almost all leaves and 
\begin{equation} \label{local_struct}
 \nu(f) = \int_{\mathbb{P}(E_x)} \Big (\int_{U} \phi(y,t) f(y,t) d\mu(y) \Big) d\gamma(t),
\end{equation}
for any bounded measurable function $f$ on $U \times \mathbb{P}(E_x)$. We write $\nu_x$ for the well-defined finite measure $\phi(x,\cdot) \gamma$ on $\mathbb{P}(E_x)$. By definition, if $f$ is a bounded measurable function on $\mathbb{P}(E)$, then
\begin{equation}
 \nu(f) = \int_X \Big( \int_{\mathbb{P}(E_x)} f(x,t) d\nu_x(t) \Big) d\mu(x).
\end{equation}

\begin{prop}\label{mes_image}
 If $\nu$ is a harmonic measure on $\mathbb{P}(E)$, then the push-forward measure $p_\ast(\nu)$ is the probability measure $\mu$ on $X$.
\end{prop}

\begin{proof}
 The function $\beta: x \in X \mapsto \nu_x(\mathbb{P}(E_x))$ is harmonic, as can be seen from the local description of $\nu$. Moreover, $\beta$ is nonnegative and integrable, of integral $1$, by Fubini theorem. By Theorem 1 of \cite{Li_L1}, the condition on the Ricci curvature of $g$ implies that $\beta$ is constant. Since $\mu(\beta) = 1$, the constant is $1$.
\end{proof}

\paragraph{The set of harmonic measures} 

The set of harmonic (probability) measures is a closed convex subset of the set of probability measures. By Proposition \ref{mes_image}, the proof of Theorem \ref{exist_harm} and Prokhorov theorem, it is a compact subset.

\begin{defn}
 A harmonic probability measure $\nu$ on the foliated space $\mathbb{P}(E)$ is \emph{ergodic} if any leaf-saturated measurable subset $S$ of $\mathbb{P}(E)$ is of mass $0$ or $1$.
\end{defn}

\begin{prop}\label{convex}
 The extremal points of the compact convex set of harmonic measures are the ergodic harmonic measures.
\end{prop}

\begin{proof}
 This is a generalization, when the foliated space may be non-compact, of Lemma F and Proposition 6 of \cite{Garn}. 
\end{proof}

\subsection{Harmonic measures and Lyapunov exponents}

Let $W(\mathbb{P}(E))$ be the space of continuous paths $\gamma:[0,+\infty[ \rightarrow \mathbb{P}(E)$, whose images are contained in a single leaf of the foliation. If $u$ is in $\mathbb{P}(E)$, then the subspace $W_u(\mathbb{P}(E))$ can be identified with the space $W_{p(u)}(X)$. Hence, $W_u(\mathbb{P}(E))$ inherits a probability measure $\mathbb{P}_u$ of \emph{foliated Brownian motion} starting at $u$. Given a probability measure $\nu$ on $\mathbb{P}(E)$, we define a probability measure $\bar{\nu}$ on $W(\mathbb{P}(E))$ by
\begin{equation}
 \bar{\nu}(f) := \int_{\mathbb{P}(E)} \Big(\int_{W_u(\mathbb{P}(E))} f(\gamma) \mathbb{P}_u(\gamma) \Big) \nu(u),
\end{equation}
where $f$ is a bounded measurable function on $W(\mathbb{P}(E))$.

As before, we define the shift $(\theta_t)$ from $W(\mathbb{P}(E))$ to itself. In \cite[\S 6]{Cand}, the following is shown:

\begin{prop}\label{inv_erg}
 The dynamical system $(W(\mathbb{P}(E)),\bar{\nu},(\theta_t))$ is invariant if and only if $\nu$ is harmonic; it is ergodic if and only if $\nu$ is ergodic.
\end{prop}

\paragraph{The cocycle in $\mathbb{P}(E)$}

Let $(\bar{H}^t)_{t \in [0,+\infty[}$ be the functions defined on $W(\mathbb{P}(E))$ by
\begin{equation}
\bar{H}^t(\gamma) := \log \frac{h(s(\gamma(t)))}{h(s(\gamma(0)))},
\end{equation}
where $s$ is an arbitrary flat lift of $\gamma$ to the tautological bundle $\mathcal{O}(-1) \rightarrow \mathbb{P}(E)$ and $h$ is the metric on $\mathcal{O}(-1)$, induced from the metric $h$ on $E \rightarrow X$. These functions satisfy the cocycle relation:
\begin{equation}\label{cocycle}
 \bar{H}^{t+t'}(\gamma) = \bar{H}^t(\gamma) + \bar{H}^{t'}(\theta_t(\gamma)).
\end{equation}

\begin{lem}
 For any probability measure $\nu$ on $\mathbb{P}(E)$, the function $\bar{H}^t$ is in $L^1(W(\mathbb{P}(E)),\bar{\nu})$. Moreover,
\begin{equation} \label{ineq_cocycle}
 \mathbb{E}_{\bar{\nu}}(\bar{H}^t) \leq \mathbb{E}_{\mu}(H^t) 
\end{equation}
\end{lem}

\begin{proof}
 One has
\begin{equation}\label{integ_triple}
 \mathbb{E}_{\bar{\nu}}(|\bar{H}^t|) = \int_X \Big( \int_{\mathbb{P}(E_z)} \Big( \int_{W_u(\mathbb{P}(E))} |\bar{H}^t(\tilde{\gamma})| \mathbb{P}_u(\tilde{\gamma}) \Big) \nu_z(u) \Big) d\mu(z).
\end{equation}

If $\gamma$ is a path in $X$, we denote by $||P_{\gamma,t_1,t_2}||$ the operator norm of the parallel transport in $E$, along $\gamma_{|[t_1,t_2]}$. It follows that
$$ |\bar{H}^t(\tilde{\gamma})| \leq \max(\log ||P_{p \circ \tilde{\gamma}, 0,t}||, \log ||P_{p \circ \tilde{\gamma}, t, 0}||).$$
In particular, $|\bar{H}^t(\tilde{\gamma})| \leq \big|\log ||P_{p \circ \tilde{\gamma}, 0,t}|| \big| + \big|\log ||P_{p \circ \tilde{\gamma}, t, 0}|| \big|$. By using this inequality in equation \eqref{integ_triple}, we obtain:
\begin{equation*}
 \mathbb{E}_{\bar{\nu}}(|\bar{H}^t|) \leq \int_X \Big(\int_{W_z(X)} \big(\big|\log ||P_{\gamma, 0,t}|| \big| + \big|\log ||P_{\gamma, t, 0}|| \big|\big) \mathbb{P}_z(\gamma) \Big) d\mu(z).
\end{equation*}
The symmetry of the heat kernel then implies that $\mathbb{E}_{\bar{\nu}}(|\bar{H}^t|) \leq 2 \mathbb{E}_{\mu}(|H^t|)$ and Theorem \ref{H1_L1} shows that $\bar{H}^t$ is in $L^1$. Inequality \eqref{ineq_cocycle} is obtained by the same proof, without the absolute values.
\end{proof}

\begin{defn}
 The (first) \emph{Lyapunov exponent} with respect to $\nu$ is $\mathbb{E}_{\bar{\nu}}(\bar{H}^1)$. It is denoted by $\lambda(\nu)$.
\end{defn}

\begin{prop}
We assume that $\nu$ is harmonic and ergodic. Then, for $\bar{\nu}$-almost every path $\tilde{\gamma}$ in $W(\mathbb{P}(E))$, the limit of $\frac{\bar{H}^t(\tilde{\gamma})}{t}$ when $t$ goes to $+\infty$ exists and is equal to $\lambda(\nu)$.
\end{prop}

\begin{proof}
 This is an application of Birkhoff ergodic theorem, since $\bar{\nu}$ is invariant and ergodic by Proposition \ref{inv_erg}.
\end{proof}

The interest of introducing these Lyapunov exponents $\lambda(\nu)$ that depend on a probability measure in $\mathbb{P}(E)$ lies in the fact that they can be computed as an integral in space.

We define a function $\phi$ on $\mathbb{P}(E)$ by
$$\phi(u) = \frac 12 \Delta_{\mathcal{F}} \log h(s(v))_{|v = u},$$
where $\Delta_{\mathcal{F}}$ is the foliated Laplacian in $\mathbb{P}(E)$ and $s$ is a local flat section of $\mathcal{O}(-1)$ in a neighborhood of $u$. 

\begin{prop} \label{form_spat}
We assume that $\phi$ is bounded on $\mathbb{P}(E)$. Then, for any harmonic measure $\nu$ on $\mathbb{P}(E)$,
$$\lambda(\nu) = \int_{\mathbb{P}(E)} \phi(u) d\nu(u).$$
\end{prop}

\begin{proof}
 Since $\nu$ is harmonic, we have that $\lambda(\nu) = \frac 1t \mathbb{E}_{\bar{\nu}}(\bar{H}^t)$ for any positive $t$. Hence,
$$\lambda(\nu) = \frac 1t \int_{\mathbb{P}(E)} \Big( \int_{\mathcal{L}_u} p^{\mathcal{L}_u}(t,u,v) \log h(s_u(v)) dv \Big) d\nu(u),$$
where $\mathcal{L}^u$ is the leaf passing through $u$ and $s_u$ is a fixed section of $\mathcal{O}(-1)$ over $\mathcal{L}^u$ such that $h(s_u(u)) = 1$.

Writing $g(t,u) = \int_{\mathcal{L}_u} p^{\mathcal{L}_u}(t,u,v) \log h(s_u(v)) dv$, $\lambda(\nu)$ is thus equal to $\frac{d}{dt}_{|t=0} \int_{\mathbb{P}(E)} g(t,u) \nu(u)$. We want to exchange the derivative and the integral. We claim that for short time $s$, the equality $\frac{d}{dt}_{|t = s} g(t,u) = \mathbb{E}_u^s (\phi)$ holds. This would immediately follows from Dynkin's formula \eqref{Dynkin} if the function $\log h(s_u(v))$ had compact support; this is not true in general but the equality can be proved using estimates on the Brownian motion, as in the proof of Theorem \ref{H1_L1}.

Since by assumption $\phi$ is bounded, one can derive under the integral. It shows that:
\begin{eqnarray*}
 \lambda(\nu) &=& \int_{\mathbb{P}(E)} \frac{d}{dt}_{|t=0} g(t,u) d\nu(u) \\
&=& \int_{\mathbb{P}(E)} \phi(u) d\nu(u).
\end{eqnarray*}
This concludes the proof.
\end{proof}

\paragraph{Criterion of equality} We are primarily interested in the Lyapunov exponent $\lambda$ of $(E,D)$; it is thus necessary to determine a criterion for the equality $\lambda = \lambda(\nu)$. We follow closely \cite[\S 1]{Furm}.

\begin{defn}
 Let $\Gamma$ be a discrete group. A representation $\rho: \Gamma \rightarrow GL(n,\C)$ is \emph{strongly irreducible} if there does not exist a finite union $W = L_1 \cup \dots \cup L_k$ of proper subspaces $L_i \subset \C^n$ such that $\rho.W = W$. Equivalently, we ask that the restriction of $\rho$ to any finite index subgroup is irreducible. 

We will also say that a flat bundle $(E,D)$ is \emph{(strongly) irreducible} if its monodromy is.
\end{defn}

\begin{thm}\label{equal_lambda}
If the flat vector bundle $(E,D)$ is strongly irreducible, then the equality $\lambda = \lambda(\nu)$ of Lyapunov exponents holds, for any harmonic measure $\nu$. 
\end{thm}

\begin{proof}
Let $\gamma$ be a random path in $W(X)$. Let $[u_1],\dots,[u_n]$ be $n$ random independent points in $\mathbb{P}(E)_{\gamma(0)}$ chosen, with probability $\nu_{\gamma(0)}$. By the following lemma, these points give a projective basis of $\mathbb{P}(E)_{\gamma(0)}$ with probability $1$. We write $u_i$ for a unit vector on the line $[u_i]$ in $E_{\gamma(0)}$. There is a positive constant $C$ such that the inequalities
\begin{equation*}
 h(P_{\gamma,0,t}.u_i) \leq ||P_{\gamma,0,t}|| \leq C.\max_i \big(h (P_{\gamma,0,t}.u_i )\big),
\end{equation*}
hold for any positive time $t$. Taking the logarithm and dividing by $t$ gives the following inequality of cocycles:
\begin{equation*}
 \bar{H}^t(\tilde{\gamma}_i) \leq H^t(\gamma) \leq \frac {\log C}{t} + \max_i \big(\bar{H}^t(\tilde{\gamma}_i)\big),
\end{equation*}
where $\tilde{\gamma}_i$ is the flat lift of $\gamma$ starting at $[u_i]$. If $\nu$ is ergodic, then with probability one, the extreme terms tend to $\lambda(\nu)$ and the middle term to $\lambda$; this concludes the proof in this case. For the general case, one can apply Proposition \ref{convex}, noticing that $\lambda(\nu)$ is an affine function in the set of harmonic measures $\nu$.
\end{proof}

\begin{lem}
 Let $\nu$ be a harmonic measure on $\mathbb{P}(E)$. Let $x$ be a point in $X$ and let $u_1,\dots,u_n$ be $n$ independent random points in $\mathbb{P}(E_x)$, chosen with probability $\nu_x$. Then, with probability one, $(u_1,\dots,u_n)$ is a projective basis of $\mathbb{P}(E_x)$.
\end{lem}

\begin{proof}
 If not, there is a proper subspace $W$ of $\mathbb{P}(E_x)$ such that $\nu_x(W) > 0$. One can choose such a $W$ of minimal dimension. Let $\Gamma$ be the subgroup of $\pi_1(X,x)$ of paths stabilizing $W$ by parallel transport and let $X^\Gamma$ be the covering space of $X$ with fundamental group $\pi_1(X)/\Gamma$. On the universal cover $\tilde{X}$, we consider the function $f(y) = \nu_y(W_y)$, where $W_y$ is the parallel transport of (some pullback of) $W$ in $\tilde{X}$. By definition of $\Gamma$, $f$ is invariant under $\Gamma$, hence $f$ is well-defined on $X^\Gamma$. We claim that it is constant.

On the one hand, the local description of a harmonic measure implies that $f$ is harmonic. On the other hand, $f$ is integrable on $X^\Gamma$. Indeed, if we transport $W$ along $\gamma$ and $\gamma'$ in $\pi_1(X,x)$, giving subspaces $W_{\gamma.x}$ and $W_{\gamma'.x}$, then whether these subspaces are equal or their intersection has $\nu_x$-measure zero, since $W$ is of minimal dimension. It follows that
\begin{eqnarray*}
 \int_{X^\Gamma} f(y)dy &=& \int_X \Big(\sum_{\gamma \in \pi_1(X,y)/\Gamma} f(W_{\gamma.y})\Big) dy. \\
&=& \int_X \Big( \bigcup_{\gamma \in \pi_1(X,y)/\Gamma} f(W_{\gamma.y})\Big) dy \\
&\leq& 1
\end{eqnarray*}
By Theorem $1$ in \cite{Li_L1}, $f$ has to be constant. This implies that $X^\Gamma$ is of finite volume; hence $\Gamma$ is of finite index in $\pi_1(X)$, contradicting the assumption of strong irreducibility of the monodromy.
\end{proof}

\begin{rem}\label{ineq_lambda}
 Without the assumption of strong irreducibility, we still have the inequality
$$\lambda(\nu) \leq \lambda$$ 
if $\nu$ is harmonic. This follows from the cocyle relation \eqref{cocycle} and the inequality \eqref{ineq_cocycle}.
\end{rem}

Assembing the results of this section, we get the following theorem:

\begin{thm} \label{thm_int_spat}
 We assume that the flat bundle $(E,D)$ over $X$ is strongly irreducible and that the function $\phi(u) = \Delta_{\mathcal{F}} \log h(s(v))_{|v = u}$ is bounded on $\mathbb{P}(E)$. Let $\mathcal{M}$ be a non-empty closed invariant subset of $\mathbb{P}(E)$. There exists a harmonic measure $\nu$ with support in $\mathcal{M}$ such that the Lyapunov exponent of $(E,D)$ is given by:
$$ \lambda = \frac 12 \int_{\mathbb{P}(E)} \phi(u) d\nu(u).$$
\end{thm}

\section{Geometric interpretation over a K\"ahler manifold}

We now assume that $X$ is a complex manifold and that the Riemannian metric $g$ is associated to a K\"ahler form $\omega$. In this section, we relate the Lyapunov spectrum to holomorphic invariants of $X$ and $E$. More precisely we show -- if $X$ is compact -- that the sum $\lambda_1 + \dots + \lambda_k$ is greater than the degree of any holomorphic subbundle of $E$ of co-rank $k$ and we discuss the equality case.

\subsection{Harmonic current}

The Laplacian $\Delta_X$ is related to the $\d \bar{\d}$ operator by the identity
\begin{equation}\label{lapl_kahl}
(\Delta_X f) \omega^d = 2d \sqrt{-1} \cdot \d \bar{\d} f \wedge \omega^{d-1},
\end{equation}
as follows from the K\"ahler identities. We recall that the volume form on $X$ is given by $\frac{\omega^d}{d!}$, where $d$ is now the complex dimension of $X$. We simply write $\omega$ for the pullback of $\omega$ on $\mathbb{P}(E)$.\\

In this setting, we associate a $(1,1)$-current to any probability measure on $\mathbb{P}(E)$.

\begin{defn}
 Let $\nu$ be a probability measure on $\mathbb{P}(E)$. Let $\alpha$ be a $(1,1)$-form with compact support on $\mathbb{P}(E)$. There exists a unique smooth function $f_\alpha$ with compact support such that the identity 
$$ \alpha \wedge \omega^{d-1} = f_\alpha \omega^d$$
holds in any leaf of $\mathbb{P}(E)$. The current $T_\nu$ is defined by
\begin{equation*}
 T_\nu(\alpha) := d\pi \int_{\mathbb{P}(E)} f_\alpha d\nu.
\end{equation*}
\end{defn}

\begin{prop}
 If $\nu$ is a harmonic measure, then $T_\nu$ is a \emph{pluriharmonic current}, i.e. it satisfies
\begin{equation*} 
 T_\nu(\d_{\mathcal{F}} \bar{\d}_{\mathcal{F}} f) = 0,
\end{equation*}
for any smooth function $f$ with compact support. Here $\d_{\mathcal{F}}$ and $\bar{\d}_{\mathcal{F}}$ are the usual differential operators $\d$ and $\bar{\d}$, in the leaves of $\mathbb{P}(E)$.
\end{prop}

\begin{proof}
 By definition of $T_\nu$ and equation \eqref{lapl_kahl}, 
$$T_\nu(\d_{\mathcal{F}} \bar{\d}_{\mathcal{F}} f) =  \frac{\pi}{2\sqrt{-1}} \int_{\mathbb{P}(E)} (\Delta_{\mathcal{F}} f) d\nu.$$ 
This vanishes if $\nu$ is harmonic by Remark \ref{rem_mes_harm}.
\end{proof}

\begin{rem}\label{rem_cohomo}
If $\nu$ is harmonic and $X$ is compact, $T_\nu$ vanishes on $\d \bar{\d}$-exact forms. Thus, $T_\nu$ defines a linear form on the Bott-Chern cohomology group $H^{1,1}_{BC}(\mathbb{P}(E),\R)$ which is equal to the de Rham cohomology group -- since $\mathbb{P}(E)$ is a compact K\"ahler manifold. It is useful to think of $T_\nu$ as a homology $(1,1)$-class.

 It is also natural to consider the $(n,n)$-current $S_\nu$, defined by 
$$S_\nu(\alpha) = \int_{\mathbb{P}(E)} f_\alpha d\nu,$$
for a $(d,d)$-form $\alpha$ whose restriction to the leaves of $\mathbb{P}(E)$ is equal to $f_\alpha \omega^n$. However, the harmonicity of $\nu$ does not imply the pluriharmonicity of $S_{\nu}$: if $\beta$ is a $(n-1,n-1)$ form, then $S_\nu(\d_{\mathcal{F}} \bar{\d}_{\mathcal{F}})$ does not vanish in general. An example arises for instance by considering a (cocompact) torsion free lattice $\Gamma$ in $PSU(1,n)$, acting holomorphically on the complex $n$-dimensional unit ball $B^n \subset \mathbb P^n $, preserving the complex hyperbolic metric (see e.g. \cite{Parker}). The complex hyperbolic compact manifold $X = \Gamma \backslash B^n $ carries a natural (projective) flat bundle $E\rightarrow X$ of rank $n+1$ whose monodromy is given by the identification of its fundamental group with $\Gamma \subset PSU(1,n) \subset PGL(n+1,\C)$. 

For any $z\in B^n$, let $\lambda_z$ be the harmonic measure on $\partial B^n$ issued from the point $z$ (the distribution of the limit in $\partial B^n$ of a Brownian trajectory starting at $z$). By homogeneity, $\lambda_z$ is a smooth measure on $\partial B^n$: the only probability measure invariant by the stabilizer of $z$ in $PSU(1,n)$. These harmonic measures are related to one another by the Poisson kernel whose expression is 
$$ P(z,u)= \frac{\big( 1 -|z|^2 \big)^n}{|1 - z\cdot \overline{u} |^{2n} },$$
see e.g. \cite{Koranyi}, page 508 (here $z\cdot u = \sum _{k=1}^{n} z_k u_k$ and $|z|^2 = z \cdot \overline{z}$). More precisely, one has 
$$  \lambda _z = P(z,\cdot) \lambda _0 .$$ 

Let $\tilde{\nu}$ be the Radon measure on $B^n \times \mathbb P ^n$ defined by 
$$  \tilde{\nu} = P (z, u) \text{vol}_g(dz) \otimes \lambda_0(du) $$
where $g$ is the complex hyperbolic metric, and $\text{vol}_g$ its volume. It is harmonic since the Poisson kernel is harmonic in the $z$ variable (in fact this is the only harmonic measure here). It is invariant by the diagonal action of $\Gamma$ on $B\times \mathbb P^n$, hence defines a finite measure on the quotient $ \mathbb P(E) = \Gamma \backslash (B\times \mathbb P^n)$. The associated current $S_\nu$ is then the quotient of the current  
$$ S_{\tilde{\nu}} (\tilde{\alpha}  ) = \int_{\partial B} \Big( \int_{B\times u } P(\cdot, u ) \tilde{\alpha} \Big) d\lambda (u) .$$ 
The Poisson kernel, while harmonic in the variable $z$ with respect to the complex hyperbolic metric, is not pluriharmonic when $n>1$. Hence, the current $S_{\nu}$ is not pluriharmonic. 
\end{rem}

We will use the following notion of degree of holomorphic bundles on non-compact manifolds.

\begin{defn}\label{def_admissible}
 A metric $h$ on a holomorphic bundle $\mathcal{E}$ is \emph{admissible} if the curvature $R^{\mathcal{E}}(h)$ of the Chern connection is bounded, with respect to the K\"ahler metric on $X$ and the metric induced by $h$ on $\End(\mathcal{E})$.
\end{defn}

The Chern form $c_1(\mathcal{E},h)$ is defined as usual by $\frac {\sqrt{-1}}{2\pi} \Tr R^{\mathcal{E}}(h)$.

\begin{defn}\label{def_degree}
 Let $h$ be an admissible metric on $\mathcal{E}$. Then the \emph{analytic degree} of $(\mathcal{E},h)$ is
$$ \deg(\mathcal{E},h) := \frac{d}{\int_X \omega^d} \int_X c_1(\mathcal{E},h) \wedge \omega^{d-1}.$$
\end{defn}

This is well-defined since the volume of $X$ is finite.

\begin{lem}\label{phi_is_bounded}
 If $h$ is an admissible metric on the flat bundle $(E,D)$ and if Assumption \ref{bound_higgs} is satisfied, then the function $\phi$ of Proposition \ref{form_spat} is bounded.
\end{lem}

\begin{proof}
 Let $(x,[v])$ be a point in $\mathbb{P}(E)$. By parallel transport on $X$, $(x,[v])$ defines a flat line subbundle $L_{[v]}$ of $E$ on a neighborhood of $x$ in $X$. The function $\phi$ satisfies
$$ \phi(x,[v]) \omega^d = 2d \sqrt{-1} \cdot \d \bar{\d} \log h(s(y))_{|y = x} \omega^{d-1},$$
where $s$ is a flat section of $L_{[v]}$ in a neighborhood of $x$. Hence, it is sufficient to prove that the Chern curvature of $L_{[v]}$ is bounded, uniformly in $x$ and $[v]$.

If $R^L(h)$ and $R^E(h)$ are the Chern curvatures of $L$ and $E$, it is well known that the equation
\begin{equation}\label{eqn_curv}
 R^L(h) = R^E(h)_{|L} - A \wedge A^\ast
\end{equation}
holds. Here, the restriction is taken with respect to the orthogonal decomposition $E = L_{[v]} \oplus L_{[v]}^\perp$ and $A \in \mathcal{C}^\infty\big(\Lambda^{1,0} \Hom(L_{[v]},L_{[v]}^\perp)\big)$ is such that $As$ is the orthogonal projection of $\nabla^{Ch,E} s$ on $L_{[v]}^\perp$, for a smooth section $s$ of $L_{[v]}$.

We claim that $A = -2 pr_{L_{[v]}^\perp} \circ \alpha^{1,0}$, where $pr_{L_{[v]}^\perp}$ is the orthogonal projection from $E$ to ${L_{[v]}^\perp}$ and $\alpha^{1,0}$ is the $(1,0)$-part of $\alpha$, defined before Assumption \ref{bound_higgs}. Indeed, the Chern connection on $E$ is given by
$$ \nabla^{Ch,E} = D - 2\alpha^{1,0},$$
as can be easily checked. Computing $A$ with a flat section $s$ of $L_{[v]}$ gives the claim.

Hence, both terms on the right hand side of equation \eqref{eqn_curv} are bounded if $h$ is admissible and Assumption \ref{bound_higgs} is satisfied. This concludes the proof.
\end{proof}

Let $c_1(\mathcal{O}(1),h)$ be the Chern form of the anti-tautological line bundle over $\mathbb{P}(E)$, with respect to the metric $h$. We can reinterpret Theorem \ref{form_spat} as follows. 

\begin{prop}\label{Lyap_courant}
Let $\nu$ be a probability measure on $\mathbb{P}(E)$. If the metric $h$ on $E$ is admissible, then
\begin{equation}
\lambda(\nu) = T_\nu(c_1(\mathcal{O}(1)),h).
\end{equation}
\end{prop}

\begin{proof}
 Let $u$ be a point in $\mathbb{P}(E)$. On a neighborhood of the leaf $\mathcal{L}_u$, we consider a flat section $s$ of $\mathcal{O}(-1)$. The Chern form $c_1(\mathcal{O}(1),h)$ is given by
$$ c_1(\mathcal{O}(1),h) = \frac {\sqrt{-1}}\pi \d_{\mathcal{F}} \bar{\d}_{\mathcal{F}} \log h(s),$$
in the directions tangent to $\mathcal{L}^u$. By Lemma \ref{phi_is_bounded}, the function $\phi(u) = \Delta_{\mathcal{F}} \log h(s(x))_{|x=u}$ is bounded. Theorem \ref{thm_int_spat} and equation \eqref{lapl_kahl} give:
\begin{eqnarray*}
 \lambda(\nu) &=& \frac 12 \int_{\mathbb{P}(E)} \Delta_{\mathcal{F}} \log h(s(v))_{|v = u} \nu(u) \\
&=& \frac 12 \cdot \frac{2d\sqrt{-1}}{d\pi} \cdot  T_\nu(\d_{\mathcal{F}} \bar{\d}_{\mathcal{F}} \log h(s)) \\
&=& - \frac{\sqrt{-1}}{\pi} \cdot \pi \sqrt{-1} \cdot T_\nu(c_1(\mathcal{O}(1),h)).
\end{eqnarray*}
This concludes the proof.
\end{proof}

\subsection{Relation with holomorphic subbundles}

From now on, we restrict to the case where $X$ is a compact K\"ahler manifold; see Subsection \ref{non_comp} for a discussion on the non-compact case.

Let $F$ be a holomorphic subbundle of $E$ of co-rank $1$. The general case will be treated in Subsection \ref{subsec_codim}. Over $\mathbb{P}(E)$, we consider the following three line bundles:
\begin{itemize}
 \item the anti-tautological line bundle $\mathcal{O}(1)$;
 \item the line bundle $\mathcal{O}([\mathbb{P}(F)])$ associated to the divisor $\mathbb{P}(F)$;
 \item the pullback of the line bundle $E/F$ over $X$.
\end{itemize}

\begin{lem}\label{isom_line}
 The line bundle $\mathcal{O}([\mathbb{P}(F)])$ is isomorphic to $\mathcal{O}(1) \otimes E/F$.
\end{lem}

\begin{proof}
Let $s$ be a local non-vanishing section of $\mathcal{O}(1)$. There is a unique local section $s^\ast$ of $\mathcal{O}(-1)$ such that $s(s^\ast) = 1$. Outside $\mathbb{P}(F)$, $\mathcal{O}(-1)$ is isomorphic to $E/F$ and $s^\ast$ can be thought as a section $u$ of $E/F$. The section $s \otimes u$ of $\mathcal{O}(1) \otimes E/F$ is well-defined outside $\mathbb{P}(F)$.

Locally, $X$ is an open set in $\C^d$ with coordinates $(x_1,\dots,x_d)$, $E$ is trivial with coordinates $(u_1,\dots,u_n)$ and $F$ is given by $u_n = 0$. Let $(x,[v])$ be the coordinates in a neighborhood of a point in $\mathbb{P}(F)$. We can assume that $v_1$ is non zero in this neighborhood. Then, if $w$ is a local section of $\mathcal{O}(-1)$, $s(x,[v])(w) = w_1$ and $u(x,[v]) = (1,v_2/v_1,\dots,v_n/v_1) \mod F$ can be taken for the local sections of $\mathcal{O}(1)$ and $E/F$. We see that $s \otimes u$ vanishes on $\mathbb{P}(F)$ with order $1$.
\end{proof}

\begin{defn}
 The \emph{dynamical degree} of $F$, with respect to a probability measure $\nu$ on $\mathbb{P}(E)$ is the quantity
$$ \delta_{F,\nu} := T_\nu\big(c_1(\mathcal{O}([\mathbb{P}(F)]))\big).$$
\end{defn}

\begin{thm}\label{fund_form}
Let $\nu$ be a probability measure on $\mathbb{P}(E)$. The formula
$$\lambda(\nu) = \delta_{F,\nu} + \pi \cdot \deg(F) $$
holds.
\end{thm}

\begin{proof}
 By Lemma \ref{isom_line}, the following equality of Chern forms holds:
$$ c_1(\mathcal{O}(1)) = c_1(\mathcal{O}([\mathbb{P}(F)])) - c_1(E/F).$$
By Proposition \ref{Lyap_courant}, $T_\nu\big(c_1(\mathcal{O}(1))\big)$ is equal to $\lambda(\nu)$. We claim that $T_\nu(c_1(E/F))$ equals $-\pi \cdot \deg(F)$.

Indeed, $T_\nu(c_1(E/F))$ is by definition equal to $d\pi \int_{\mathbb{P}(E)} f d\nu$, where $f$ satisfies $f \omega^d = c_1(E/F) \wedge \omega^{d-1}$ on the leaves of $\mathbb{P}(E)$. Since both $\omega$ and $c_1(E/F)$ come from $X$, $f$ is constant in the fibers of $p: \mathbb{P}(E) \rightarrow X$ and
$$ T_\nu(c_1(E/F)) = \frac{d\pi}{\int_X \frac{\omega^d}{d!}} \int_X f \frac{\omega^d}{d!}.$$
The claim then follows from the definition we gave for the degree.

It follows that $\delta_{F,\nu}$ satisfies the above formula.
\end{proof}

\subsection{The dynamical degree}

Now we give a geometric interpretation of $\delta_{F,\nu}$. In fact, we define a geometric intersection between the current $T_\nu$ associated to a probability measure $\nu$ and a general hypersurface $Y$ in $\mathbb{P}(E)$.

The general definition is technically involved. However, when $Y$ is transverse to $\mathcal F$, the idea is quite simple. In that case, one can find foliated coordinates $(z,t)$ where $\mathcal F$ is given by $t=cst$, while $Y$ is given by $z_1= 0$. If $T  = \int \phi (z,t) dm(t)$ is the desintegration of the harmonic current as an integral of harmonic functions along the leaves, the restriction of $T$ to $Y= \{z_1= 0\}$ has a well-defined meaning: indeed, the functions $\phi(.,t)$ extends as harmonic functions at $z= 0$ for $m$-almost any $t$. In particular, the measure $\omega ^{d-1} \otimes m$ is well-defined on $Y$. The total mass of this measure is the intersection $T_\nu \cap [Y]$. In general, when no transversality holds, it is better to use a partition of unity in order to take into account the multiplicities. 

Let $Y$ be a hypersurface in $\mathbb{P}(E)$ such that $Y$ contains no germ of a leaf of the foliation $\mathcal{F}$. Let $(U_i,\chi_i)_{i \in I}$ be a partition of unity of $X$; we assume that the open sets $U_i$ are simply connected so that $p^{-1}(U_i)$ is diffeomorphic by parallel transport to the product $U_i \times \mathbb{P}(E_{x_i})$, where $x_i$ is a point in $U_i$. If $\alpha$ is a smooth $(1,1)$-form with compact support in $U_i$, then
$$ T_\nu(\alpha) = \frac{\pi}{\vol(X)} \int_{\mathbb{P}(E_{x_i})} d\gamma_i(t) \Big(\int_{\{t\} \times U_i} \phi_i(z,t)\alpha(z,t) \wedge \frac{\omega^{d-1}}{(d-1)!} \Big),$$
with the notations of equation \eqref{local_struct}.

\begin{defn}
 Let $f_i$ be an equation of $Y$ in $U_i$. The \emph{geometric intersection} of $Y$ and $T_\nu$ is defined by
\begin{align*}
  &T_\nu \cap [Y] := \frac {\sqrt{-1}}{\vol(X)} \\
 &\sum_i \int_{\mathbb{P}(E_{x_i})} d\gamma_i(t) \Big(\int_{\{t\} \times U_i} \phi_i(z,t)\chi_i(z,t) \d \bar{\d} \log |f_i(z,t)| \wedge \frac{\omega^{d-1}}{(d-1)!} \Big),
\end{align*}
where the inner integral is understood in the sense of currents. This is well-defined since, by assumption on $Y$, the function $f_i(\cdot,t)$ does not identically vanish on $\{t\} \times U_i$.
\end{defn}

The equation $f_i(\cdot,t) = 0$ gives a divisor $Y_{i,t}$ in $\{t\} \times U_i$, for any fixed $t$ in the fiber $\mathbb{P}(E_{x_i})$. We write
$$Y_{i,t} = \sum_k n_{t,i,k} Z_{t,i,k},$$ 
where $n_{i,k}$ is a positive integer and $Z_{t,i,k}$ is an analytic hypersurface in the local leaf $\{t\} \times W_i$. Then, the geometric intersection is denoted by $T_\nu \cap [Y]$ and is equal to 
\begin{equation} \label{geom_inter}
\frac{\pi}{\vol(X)} \sum_i \int_{\mathbb{P}(E_{x_i})} d\gamma_i(t) \Big( \sum_k n_{t,i,k} \int_{Z_{t,i,k}} \phi_i(z,t)\chi_i(z,t) \frac{\omega^{d-1}}{(d-1)!} \Big),
\end{equation}
thanks to the Lelong-Poincaré formula.

In the following, we assume that $\mathbb{P}(F)$ satisfies the weak condition of containing no germ of a leaf. If not, then there is a section of $F$ whose parallel transport always stays in $F$ by the analytic continuation principle.

\begin{thm}\label{thmconj_princ}
 The geometric intersection $T_\nu \cap [\mathbb{P}(F)]$ is finite and equals the dynamical degree $\delta_F$. In particular, $\delta_F \geq 0$ with equality if, and only if $\mathbb{P}(F)$ does not encounter the support of the current $T_\nu$.
\end{thm}

\begin{proof}
The notations are as above, with $Y = \mathbb{P}(F)$. The equations $f_i$ define a global section $s$ of the line bundle $L := \mathcal{O}([\mathbb{P}(F)])$. Over $U_i$, the following equality of currents holds:
$$ \frac{\sqrt{-1}}{\pi} \d \bar{\d} \log h(s) = \frac{\sqrt{-1}}{\pi} \d \bar{\d} \log |f_i| - c_1(L,h).$$
Then, $T_{\nu}(c_1(L,h)) = \sum_i T(\chi_i c_1(L,h))$; hence it is equal to $\frac{\pi}{\vol(X)}$ times
\begin{equation} \label{geom_def}
 \sum_i \int_{\mathbb{P}(E_{x_i})} d\gamma_i(t) \Big(\int_{\{t\} \times U_i} \phi_i(z,t) \chi_i(z,t) c_1(L,h) \\ \wedge \frac{\omega^{d-1}}{(d-1)!}\Big)
\end{equation}
\begin{multline*}
= \frac{\sqrt{-1}}{\vol(X)} \sum_i \int_{\mathbb{P}(E_{x_i})} d\gamma_i(t) \Big(\int_{\{t\} \times U_i} \phi_i(z,t) \chi_i(z,t) \big(\d \bar{\d} \log |f_i| - \\ \d \bar{\d} \log |h(s)|\big) \wedge \frac{\omega^{d-1}}{(d-1)!}\Big).
\end{multline*}

The first term is the geometric intersection $T_\nu \cap [\mathbb{P}(F)]$. We claim that the other term vanishes; intuitively this follows from the fact that $T_\nu$ is a pluriharmonic current and we apply it to a $\d \bar{\d}$-exact current. 

Here is a formal proof. Let $\psi$ be a smooth function on $\mathbb{P}(E)$ which equals $1$ outside a neighborhood $W$ of $\mathbb{P}(F)$ and $0$ on a (smaller) neighborhood $V$. Since $\d \bar{\d} (\psi\log |h(s)|)$ is a smooth $(1,1)$-form, one gets $T_\nu(\d \bar{\d} (\psi\log |h(s)|)) = 0$ by pluriharmonicity of $T_\nu$. On the other hand, the integral

$$\int_{\mathbb{P}(E_{x_i})} d\gamma_i(t) \Big(\int_{\{t\} \times U_i} \phi_i(z,t) \chi_i(z,t) \big(\d \bar{\d} ((1-\psi)\log |h(s)|) \big) \wedge \omega^{d-1}\Big),$$
understood in the sense of currents, tends to $0$ when the measure of the neighborhood $V$ tends to $0$ since $\log |h(s)|$ is a locally integrable function.

This shows that $\delta_{F,\nu} = T_\nu \cap [\mathbb{P}(F)]$. The assertions of nonnegativity and positivity then follow from equation \eqref{geom_inter}.
\end{proof}

\subsection{Higher codimension} \label{subsec_codim}

The setting is the same as before, except that $F$ has codimension $k$ in $E$. We explain how to obtain from $F$ informations on the partial sum of Lyapunov exponents $\lambda_1 + \dots + \lambda_k$.

Let $F^o$ be the annihilator of $F$ in $E^\ast$. The exterior product $\Lambda^k F^o$ can be thought as a line bundle $L$ in $\Lambda^k (E^\ast)$; let $\hat{F}$ be the annihilator of $L$ in $\Lambda^k E$, that we identify to the dual of $\Lambda^k (E^\ast)$.

\begin{lem}\label{equal_degree}
The following equality of degrees holds:
 $$\deg(F) = \deg(\hat{F}).$$
\end{lem}

\begin{proof}
 From the exact sequence
$$ 0 \rightarrow F^o \rightarrow E^\ast \rightarrow F^\ast \rightarrow 0,$$
we get that $\deg(F^o) + \deg(F^\ast) = \deg(E^\ast) = 0$ since $E$ is flat. Hence, $\deg(F^o) = \deg(F)$. In the same way, the exact sequence
$$ 0 \rightarrow \hat{F} \rightarrow \Lambda^k E \rightarrow (\Lambda^k F^o)^\ast \rightarrow 0,$$
implies that $\deg(\Lambda^k F^o) = \deg(\hat{F})$. Since $F^o$ is a vector bundle of rank $k$, we also have $\deg(F^o) = \deg(\Lambda^k F^o)$. This concludes the proof.
\end{proof}

By the Plücker embedding, the bundle $\Gr(k,E)$ of plans of dimension $k$ in $E$ is a subbundle of the projectivized bundle $\mathbb{P}(\Lambda^k E)$. We have the following geometric description:

\begin{prop} \label{descr_Fhat}
 The intersection $\Gr(k,E) \bigcap \mathbb{P}(\hat{F})$ in $\mathbb{P}(\Lambda^k E)$ is the set of $k$-planes in $E$ that intersect $F$ non trivially.
\end{prop}

\begin{proof}
 This is a statement in linear algebra. Let $G$ be a $k$-plane, with basis $v_1,\dots,v_k$. Let $f_1^\ast,\dots,f_k^\ast$ be a basis of $F^o$. In the Plücker embedding, $G$ is identified with the point $[v_1 \wedge \dots \wedge v_k]$ in $\mathbb{P}(\Lambda^k E)$. Thus, by definition of $\hat{F}$, $G$ is in $\hat{F}$ if and only if $(f_1^\ast \wedge \dots \wedge f_k^\ast)(v_1 \wedge \dots \wedge v_k) = 0$. This is equivalent to the non-invertibility of the matrix $(f_i^\ast(v_j))_{i,j}$, hence to the existence of a non-trivial linear combination $v = \sum \lambda^j v_j$ such that $f_i^\ast(v) = 0$, for every $i$. This happens if and only if $v$ is in $F$; thus such a $v$ exists if and only if $G$ intersects $F$ non trivially.
\end{proof}

\begin{rem}
 If $k = 1$, then $\mathbb{P}(\hat{F}) = \mathbb{P}(F)$, in accordance with previous subsections.
\end{rem}

We now consider a harmonic measure $\nu$ on $\mathbb{P}(\Lambda^k E)$. From Remark \ref{supp_harm} and Proposition \ref{convex}, we can assume that $\nu$ is supported on the Grassmannian $\Gr(k,E)$. From Theorem \ref{fund_form}, we known that the Lyapunov exponent $\lambda(\nu)$ satisfies:
$$ \lambda(\nu) = \delta_F + \pi \cdot \deg(F),$$
where we write $\delta_F$ for $\delta_{\hat{F},\nu}$.
Moreover, from Definition \ref{def_higher}, Remark \ref{ineq_lambda} and Theorem \ref{equal_lambda}, the inequality
$$ \lambda_1 + \dots + \lambda_k \geq \lambda(\nu) $$
holds, with equality if the monodromy $\Lambda^k \rho: \pi_1(X) \rightarrow GL(\Lambda^k \C^n)$ is strongly irreducible -- we say that $(E,D)$ is \emph{strongly $k$-irreducible}. By Theorem \ref{thmconj_princ}, we get:

\begin{prop}
 The Lyapunov exponents of $(E,D)$ satisfy the inequality
$$ \lambda_1 + \dots + \lambda_k \geq \pi \cdot \deg(F).$$
Moreover, if $(E,D)$ is strongly $k$-irreducible, then the equality
\begin{equation} \label{fund_equal}
 \lambda_1 + \dots + \lambda_k = \pi \cdot \deg(F)
\end{equation}
holds if and only if the the support of the harmonic current $T_\nu$ does not intersect $\mathbb{P}(\hat{F})$. 
\end{prop}

Since the support of $T_\nu$ can be assumed to be contained in the Grassmannian $\Gr(k,E)$, the criterion of equality can be used as follows:

\begin{prop} \label{fund_prop}
We assume that $(E,D)$ is strongly $k$-irreducible. If there exists a closed invariant subset $\mathcal{M}$ of $\Gr(k,E)$ such that any $k$-plane $G$ in $\mathcal{M}$ intersects $F$ trivially, then the equality \eqref{fund_equal} holds.
\end{prop}

\begin{proof}
 This follows from Remark \ref{supp_harm} and Proposition \ref{descr_Fhat}.
\end{proof}

\subsection{On the non-compact case} \label{non_comp}

The case where $X$ is non-compact causes a lot of technical complications. In order to simplify the discussion, we assume that $X$ is the complement of a normal crossing divisor in a smooth projective variety $\bar{X}$. Along the divisor, we choose for the metric on $X$ the product of hyperbolic metrics on the pointed disk $\mathbb{D}^\ast$ and euclidean metrics on the disk $\mathbb{D}$; see for instance \cite{mo_nilp}, subsection 4.1. This metric satisfies of course Assumption \ref{bound_geom}. If $(E,D)$ is a flat bundle over $X$, the local monodromies are given by $k$ commuting matrices, where $k$ is the number of local equations of the normal crossing divisor. 

If moreover $(E,D)$ underlies a variation of complex Hodge structures, then the local monodromies have eigenvalues of modulus one, by a theorem of Borel. Moreover, the canonical metric $h$ on $E$ satisfies Assumption \ref{bound_higgs}, thanks to some curvature properties of the period domains. Hence, Proposition \ref{Lyap_courant} applies to this situation.

The troubles come with the holomorphic subbundle $F$. There are two possible definitions for what we have called the dynamical degree of $F$ : the analytic one
$$\delta_{F,h,\nu} := T_\nu(c_1(\mathcal{O}([\mathbb{P}(F)]),h))$$
or the geometric one as in equation \eqref{geom_def}. In order to have the results of this section, it would be nice to show that the two definitions coincide. But it is already unclear what are the conditions of bounded geometry to impose on $F$, so that the geometric definition makes sense. This needs to be clarify in the future. \\

In the next section, we summarize the results already contained in the literature, concerning the equality between sum of Lyapunov exponents and the degree of holomorphic subbundles. We give another proof of these results in the case where $X$ is a compact K\"ahler manifold (of arbitrary dimension). In Subsection \ref{hypergeom}, it is assumed that our results are also valid above the sphere minus three points, for the vector bundles that are considered.

\section{Applications}

The relation between Lyapunov exponents and degrees of holomorphic subbundles has been first observed for a flat bundle carrying a variation of Hodge structures of weight $1$. In the first subsection, we slightly generalize this example and explain how it reduces to a problem in linear algebra. In the second subsection, we discuss about a basic example where we get the equality \eqref{fund_equal}, though the monodromy representation is Zariski dense. In the third subsection, we study the flat bundles that come from the hypergeometric equation and suggest a way to prove the observed phenomena.

\subsection{Families of Hodge structures}

\begin{defn}\label{CFHS}
 Let $X$ be a complex manifold. A \emph{family of complex Hodge structures} of weight $w$ over $X$ is the datum of a complex flat vector bundle $(E,D)$, a non-degenerate flat Hermitian form $h$ on $E$ and an $h$-orthogonal decomposition 
$$ E = \bigoplus^{\perp}_{0 \leq p \leq w} E^p $$
such that, writing $F^p = \oplus_{q \geq p} E^q$, the following conditions are satisfied:
\begin{enumerate}
 \item the decreasing filtration $F^\bullet$ varies holomorphically;
 \item $h$ is positive definite on $E^p$ if $p$ is even and negative definite if $p$ is odd.
\end{enumerate}
\end{defn}

We emphasize that we do not consider \emph{variations} of Hodge structures -- where the axiom of Griffiths' transversality is added -- since it will not be used in the following. The vector $(\dim E^0,\dots,\dim E^w)$ is called the \emph{type} of the family. By changing the signs of $h$ on $E^p$, for odd $p$, we define a Hermitian metric $\hat{h}$ on $E$. We call it the \emph{harmonic metric}.

The most important examples come by looking at the cohomology of a family of compact K\"ahler manifolds (see e.g. \cite{Voi1}). The flat bundle then has a real -- and in fact integral -- structure. 

\begin{defn}\label{RFHS}
 Let $X$ be a complex manifold. A \emph{family of real Hodge structures} of weight $w$ over $X$ is the datum of a real flat vector bundle $(E_\R,D)$, a non-degenerate bilinear form $Q$ which is orthogonal for even $w$ and symplectic for odd $w$, and a decomposition $E_\C := E_\R \otimes_\R \C = \oplus_{p=0}^w E^p$ such that:
\begin{enumerate}
 \item $\overline{E^p} = E^{w-p}$;
 \item Writing $\epsilon = 1$ if $w$ is even and $i$ if $w$ is odd, the form $h(u,v) := \epsilon Q_\C(u,\bar{v})$ is Hermitian and we ask that $(E_{\mathbb{C}} = \oplus_{p=0}^w E^p,D,h)$ is a family of complex Hodge structures.
\end{enumerate}
\end{defn}

We now assume that $X$ is a compact K\"ahler manifold.

\begin{prop}
Let $(E = E^0 \oplus E^1,D,h)$ be a family of complex Hodge structures of weight $1$ and type $(h^0,h^1)$ over $X$. We assume that $h^0 \leq h^1$ and that the monodromy is strongly $h^0$-irreducible. Then
$$ \sum_{k=1}^{h^0} \lambda_k = \pi \cdot \deg(E^1).$$
\end{prop}

\begin{proof}
 The vector bundle $E^1$ is holomorphic of co-rank $h^0$. On the Grassmannian bundle $\Gr(h^0,E)$, we consider the subset $\mathcal{M}$ of $h$-isotropic $h^0$-planes. This is a closed invariant subset of $\Gr(h^0,E)$ since $h$ is flat. Moreover, since $E^1$ is positive definite for $h$, it cannot intersect an isotropic plane. We conclude by applying Proposition \ref{fund_prop}.
\end{proof}

Such arguments can also be used in greater weight, as was observed in \cite{Filip}. We consider a family of real Hodge structures $(E,D)$ of weight $2$ and type $(1,k,1)$ over $X$; such situations arise when looking at the second cohomology group of families of families of $K3$ surfaces, see \cite{Filip}.

\begin{prop}\label{thm_w2}
 Writing $E_{\C} = E^0 \oplus E^1 \oplus E^2$, one has
 $$\lambda_1 = \pi \cdot \deg(E^2).$$
 \end{prop}

\begin{proof}
 In the projective bundle $\mathbb{P}(E^{\ast}_\C) \cong \Gr(k+1,E_\C)$, we consider the subset $\mathcal{M}$ of $k+1$ planes $P$ in $E_\C$ whose orthogonal is an isotropic real line. We claim that $E^2$ cannot encounter any such plane $P$. Indeed, since $\dim E_2 = 1$, this would imply that $E_2 \subset P$. Writing $L$ for the orthogonal of $P$, $L$ is in particular orthogonal to $E^2$, hence lives in $E^0 \oplus E^1$. Since $L$ is real, it has to live in $E^1$. This is a contradiction since there is no isotropic line in $E^1$.

From Proposition \ref{fund_prop}, we get that
$$ \lambda_ 1 + \sum_{i=2}^{k+1} \lambda_i = \pi \cdot \deg(E^2).$$
We conclude  by remarking that $\sum_{i=2}^{k+1} \lambda_i = 0$, by the symmetry of the Lyapunov spectrum: cf. Proposition \ref{symm_lyapunov}.
\end{proof}

In both proofs, the leaf-invariant closed subset $\mathcal{M}$ that we construct is not only invariant by the monodromy: it is also invariant under its real Zariski closure. This is why we consider that these situations can be reduced to linear algebra. The situation will be very different in the following subsections.

\subsection{An example with Zariski dense monodromy}

Let $\Gamma$ be a torsion-free finitely generated Kleinian group: that is, $\Gamma$ is a discrete subgroup of $\text{PSL}(2,\C)$. We have an action of $\Gamma$ on the sphere $\mathbb{P}^1(\C)$; we write $\Lambda(\Gamma)$ for the limit set of $\Gamma$ and $\Omega(\Gamma) := \mathbb{P}^1(\C) - \Lambda(\Gamma)$ for the discontinuity set. By Ahlfors finiteness theorem \cite{Ahlf_fin}, the quotient $S := \Omega(\Gamma)/\Gamma$ has a finite number of connected components $S_i$ and each $S_i$ is a compact Riemann surface with a finite number of points removed. We assume that some $S_i$ is compact, for simplicity.

Let $\Omega_i$ be the inverse image of $S_i$ in the projection $\Omega(\Gamma) \rightarrow \Omega(\Gamma)/\Gamma$. The universal cover $\tilde{S}_i$ projects on $\Omega_i$, giving a holomorphic map $\phi: \tilde{S}_i \rightarrow \mathbb{P}^1(\C)$, which is $\pi_1(S_i)$-equivariant for the canonical representation $\rho: \pi_1(X) \rightarrow \Gamma$. To the map $\phi$ corresponds a projective bundle $\mathbb{P}$ of rank $1$ over $S_i$, with a holomorphic section $L$. We claim that there is a harmonic measure $\nu$ such that the dynamical degree $\delta_{L,\nu}$ vanishes.

Indeed, the closed subset $\Lambda(\Gamma)$ in $\mathbb{P}^1(\C) \cong \mathbb{P}_x$ is invariant by the monodromy. If $\mathcal{M}$ is the union of leaves in $\mathbb{P}$ passing through $\Lambda(\Gamma)$, it is a closed invariant subset of $\mathbb{P}(E)$. Hence, there exists a harmonic measure $\nu$ with support on $\mathcal{M}$. The line bundle $L$ does not encounter $\mathcal{M}$ since the map $\phi$ takes its values in $\Omega(\Gamma)$. This proves the claim.

On the other hand, it is important to remark that the image of the monodromy will in general be dense in $\text{PSL}(2,\C)$ for the real Zariski topology. This is for instance the case for quasi-Fuchshian groups (which are not Fuchsian) or Schottky groups $\Gamma$.

\subsection{On the hypergeometric equation} \label{hypergeom}

In this subsection, we assume that our results are valid in the non-compact case ; see Subsection \ref{non_comp} for more details.\\

Let $X = \mathbb{P}^1(\C) - \{0,1,\infty\}$ with its hyperbolic metric. One can consider families of $3$-dimensional Calabi-Yau manifolds over $X$. The degree $3$ cohomology of such families gives interesting examples of variations of real Hodge structures of weight $3$ and type $(1,1,1,1)$; we write $E_\C = E^0 \oplus E^1 \oplus E^2 \oplus E^3$. Each $E^i$ is thus a complex line bundle. Following the proof of Theorem \ref{thm_w2}, we consider in $\Gr(2,E_\C)$ the subset $\mathcal{M}$ of $2$-planes in $E_\C$ that are real and isotropic. If $E^2 \oplus E^3$ did not intersect any $2$-plane in $\mathcal{M}$, then we would obtain an equality for the sum $\lambda_1 + \lambda_2$. However, this does not work:

\begin{lem} \label{lin_alg_lem}
 The $2$-plane $E^2 \oplus E^3$ always intersect in a non-trivial way some $2$-plane in $\mathcal{M}$.
\end{lem}

\begin{proof}
 This is a statement in linear algebra. We choose a orthogonal basis $(v_0,v_1,v_2,v_3)$ of $E_\C$, adapted to the decomposition $E_\C = \oplus_{i=0}^3 E^i$ such that:
\begin{itemize}
 \item $\overline{v_0} = v_3; \overline{v_1} = v_2$;
 \item $h(v_i,v_i) = (-1)^i$.
\end{itemize}
The plane $P$ generated by $v_2 + v_3$ and its conjugate $v_0 + v_1$ is real and isotropic; hence $P$ is in $\mathcal{M}$ and intersects the $2$-plane $E^2 \oplus E^3$.
\end{proof}

\paragraph{Hypergeometric cases}

Singular families of $3$-dimensional Calabi-Yau manifolds over $\mathbb{P}^1(\C)$ are studied in \cite{EnckStrat}. A table is given on page 11 of this paper and describes some numerical invariants attached to these families; there are 14 cases where the number of singularities is equal to $3$: they are called \emph{hypergeometric cases} and can be thought as smooth families over $X$. In \cite{Kont}, the Lyapunov exponents of these 14 families are computed by numerical experiments. The following has been observed: there are 7 \emph{good cases} and 7 \emph{bad cases}. 

For good cases, the sum of Lyapunov exponents $\lambda_1 + \lambda_2$ is rational (up to some normalization) and a formula involving the eigenvalues of the local monodromies (near the singularities) can be given. This does not work in bad cases. Our goal is to give some explanations of this phenomenon in the general framework of our paper.

\paragraph{Thin and thick monodromies}

In all 14 examples, the monodromy representation is Zariski dense in the symplectic group $\Sp(4,\R)$ and takes values in $\Sp(4,\Z)$. One says that the representation is \emph{thin} if its image has infinite index in $\Sp(4,\Z)$; it is \emph{thick} otherwise. In \cite{BravThom}, it was shown that 7 monodromies among the 14 are thin; it had been remarked by M. Kontsevitch that they correspond exactly to the 7 good cases of his numerical experiments.

\paragraph{Sketch of a proof}

The formula observed by M. Kontsevitch for good cases is essentially the equality \eqref{fund_equal}. We want to prove the following conjecture:

\begin{conj}
 There exists a closed invariant subset $\mathcal{M'}$ in $\mathcal{M}$ such that any $2$-plane $P$ in $\mathcal{M'}$ intersects trivially the $2$-plane $E^2 \oplus E^3$, if and only if, the monodromy representation is thin.
\end{conj}

A proof of this conjecture will explain the numerical observations of \cite{Kont}. Using Lemma \ref{lin_alg_lem}, we can give a proof of the easy direction.

\begin{prop}
 Suppose that the representation is thick. Then, $\mathcal{M}$ does not contain any stricly smaller closed invariant subset.
\end{prop}

\begin{proof}
Let $\Gamma$ denote the image of the monodromy in $\Sp(E_x)$ and let $P$ be an arbitrary real isotropic $2$-plane in $\Gr(2,E_x)$. We observe that the orbit $\Sp(E_x,\Z).P$ is dense (for the Hausdorff topology) in the set $\mathcal{M}_x$ of real isotropic $2$-planes. This is clear if $P$ is rational and is true in general using a translation. Since by assumption $\Gamma$ is of finite index in $\Sp(E_x,\Z)$, there exists a finite number of $\gamma_i$ in $\Sp(E_x,\Z)$ such that
$$ \Sp(E_x,\Z).P = \cup_{i=1}^r \Gamma.\gamma_i P.$$
Taking the closure, this gives a partition of $\mathcal{M}_x$ in a finite number of closed $\Gamma$-invariant subsets. By connectedness of $\mathcal{M}_x$, this is possible only if $\Gamma.P$ itself is already dense in $\mathcal{M}_x$. This concludes the proof of the proposition.
\end{proof}

The other direction is an interesting challenge. The idea goes as follows: we consider one of the 7 representations with thin monodromy. One has to have a good understanding of this representation in order to construct a proper closed invariant subset $\mathcal{M'}$ in $\mathcal{M}$ and then prove that the $2$-plane $E^2 \oplus E^3$ does not meet an arbitrary $2$-plane $P$ in $\mathcal{M'}$. For the first step, some ping-pong lemma arguments, as in \cite{BravThom}, should lead to a conclusion. 

It is not clear to us whether it is possible to compute things directly or if a clever argument is available for the second step.

\newpage

\bibliographystyle{alpha}
\bibliography{bib_harm.bib}

\end{document}